\documentclass[10pt,reqno,oneside]{amsart}
\usepackage{amsfonts,amssymb,latexsym,xspace,epsfig,graphics,color}
\usepackage{amsmath,enumerate,stmaryrd,xy}
\usepackage{bbm}
\usepackage{natbib}
\usepackage{subfigure}

\newcommand{\N}{\ensuremath{\mathbb{N}}}

\newcommand{\R}{\ensuremath{\mathbb{R}}}     
\newcommand{\Z}{\ensuremath{\mathbb{Z}}}    
\renewcommand{\P}{\ensuremath{P}}   

\usepackage{algorithmic}
\usepackage{algorithm}

\newtheorem{lemma}{Lemma}

\newtheorem{theo}{Theorem}
\newtheorem{ex}{Example}
\newtheorem{prop}{Proposition}

\begin{document}


\title{Rumor processes on $\N$ and discrete renewal processes}

\author{Sandro Gallo}
\address[Sandro Gallo]{Departamento de M\'etodos Estat\'\i sticos - Instituto de
  Matem\'atica - Universidade Federal do Rio de Janeiro (UFRJ)} 
\email{sandro@im.ufrj.br}

\author{Nancy L. Garcia}
\address[Nancy L. Garcia]{Departamento de Estat\'\i stica - Instituto de Matem\'atica, Estat\'\i stica e Computa\c c\~ao
  Cient\'\i fica - Universidade de
  Campinas (UNICAMP)}
\email{nancy@ime.unicamp.br}

\author{Valdivino Vargas Junior}
\address[Valdivino Vargas Junior]{Instituto de Matem\'atica e Estat\'\i stica - Universidade
  Federal de Goi\'as (UFG)} 
\email{vv.junior@mat.ufg.br}

\author{Pablo M. Rodr\'\i guez}
\address[Pablo M. Rodr\'\i guez]{Instituto de Ci\^encias Matem\'aticas e de Computa\c c\~ao -
  Universidade de S\~ao Paulo (USP)}
\email{pablor@icmc.usp.br}

\date{December 18, 2013}

\begin{abstract}
We study two rumor processes on $\N$, the dynamics of which are related to an SI epidemic model with long range transmission. 
Both models start with one spreader at site $0$ and ignorants at all the other sites of $\N$, but differ by the transmission mechanism. In one model,  the spreaders transmit the information within a random distance on their right, and in the other the ignorants take the information from a spreader within a random distance on their left. 

We obtain the probability of survival, information on the distribution of the range of the rumor and limit theorems for the proportion of spreaders. The key step of our proofs is to show that, in each model, the position of the spreaders on $\N$ can be related to a suitably chosen discrete renewal process.
\end{abstract}

\maketitle

\section{Introduction}

In the last decades many works dealt with the analysis of the phenomenon of information transmission (news or rumors) from a probabilistic point of view. The resulting stochastic models are, in general, inspired in the classical SIR, SIS and SI epidemic models. In these models, it is assumed that an infection or information spreads through a population subdivided into susceptibles, infectives, and removed individuals, who are referred to as ignorants, spreaders, and stiflers when one deal with rumor difusion processes.

In the context of the SIR epidemic model, rumor processes were introduced by \cite{daley/kendall/1965} and by \cite{maki/thompson/1973}. In such models, a finite, closed and homogeneously mixing population is considered. Spreaders try to tell the rumor to ignorants, and stiflers appear either through the meeting of spreaders or through the meeting of a spreader with a stifler. The later transitions represent the loss of interest in propagating the rumor when a spreader meets someone  that already knows the rumor. The well known results for these models are limit theorems for the remaining proportion of ignorants when there are no spreaders left in the population, that is, at the end of the process. Some generalizations of the basic models and recent results can be found, for instance, in \cite{lebensztayn/machado/rodriguez/2011a,lebensztayn/machado/rodriguez/2011b}, \cite{comets/delarue/schott/2013}, and references therein. Variations of the Maki-Thompson rumor model were  considered in several graphs. For instance, the survival of the rumor was studied in \cite{moreno/nekovee/pacheco/2004} and \cite{isham/harden/nekovee/2010} when the population is represented by a random or complex network, and in \cite{coletti/rodriguez/schinazi/2012} when the population is represented by the d-dimensional hypercubic lattice.

When the population is composed only by spreaders and ignorants, the rumor process is called SIS or SI epidemic model. In the former, a spreader may become ignorant, whereas  in the later spreaders remain in such state forever. Recent results for the SIS model, known in the probabilistic literature as the contact process, which can be of interest in the context of information spreading, can be found in \cite{berger/borgs/chayes/saberi/2005} and \cite{durrett/jung/2007}. On the other hand, one of the first SI rumor models was proposed and analyzed, in an homogeneously mixing population, by \cite{pittel/1987}. In this case, the author studied the distribution of the number of stages before everybody is informed by means of an approximation on the number of spreaders at time $t$ by a deterministic equation.\\

The purpose of this paper is to study rumor processes on $\N$, the dynamics of which are related to the process considered by \cite{pittel/1987}, but with long range transmissions. 
More precisely, we consider two long range rumor spreading models, initially introduced by \cite{junior/machado/zuluaga/2011}, called firework and reversed firework processes (FP and RFP in the sequel). Both models start with one spreader at site $0$ and ignorants at all the other sites of $\N$. The difference between them is the transmission mechanism. In the FP, each spreader transmits the information, independently, to the individuals within a random distance on its right. In the RFP, each ignorant takes the information, independently, from a spreader within a random distance on its left. 

\cite{junior/machado/zuluaga/2011} gave sufficient conditions  under which the rumor survives, or not, with positive probability. It is worth noticing that related results have been obtained recently by \cite{bertacchi/zucca/2013} and previously by \cite{athreya/roy/sarkar/2004} (the later in the context of space covering processes). 
In the present paper, we give necessary and sufficient conditions for survival of the rumor.
Our method of proof, based on a direct comparison between the rumor processes and a discrete time renewal process, allows us to obtain several  additional results. For the FP, we obtain the exact expression for the probability of survival. We also obtain information about the distribution of the range of the rumor when it dies out applying results of \cite{garsia/lamperti/1962}, \cite{bressaud/fernandez/galves/1999a} and \cite{gallo/lerasle/takahashi/2013}, all of them concerning renewal theory. For the RFP, we obtain a law of large numbers and a central limit theorem for the proportion of spreaders in a range of size $n$ as $n$ diverges. We point out that this type of results, namely limit theorems for the proportion of individuals of a certain class, has been of interest in many of the papers previously cited. Related models, and results, can be found, for instance, in \cite{kurtz/lebensztayn/leichsenring/machado/2008}, where information transmission is modeled by means of a system of random walks, or \cite{andersson/1998}, where an epidemic model is described in the framework of random graphs.

\section{Models and main results}

In what follows,  ${\bf R}=(R_{i})_{i\ge1}$ will always be  a sequence of $\N$-valued i.i.d. random variables with distribution 
$$P(R_{0}=k)=\lambda_{k}, \,\,\, k=0,1,2,\ldots$$ 
where $\lambda_0\in(0,1)$. For each value of $k\ge0$ let $\alpha_{k}:=P(R_{0}\leq k)$. 
%

\subsection{The Firework process}

Suppose that one individual is disposed at each site of $\N=\{0,1,\ldots\}$.  In this model, the spreaders transmit the information within a random distance to their right. For any $n\ge0$, let $A_n$ represents the set of individuals that have been informed at stage $n$ in the Firework process. Initially, only $0$ is a spreader, and thus $A_0=\{0\}$. Then, the sequence  $(A_n)_{n\ge1}$ is defined recursively through
\[A_n:=\{i\in\N:\,\,\textrm{there exists} \,\,j\in A_{n-1}\,\,\textrm{such that}\,\,i\in\{j,\ldots,j+R_j\}\}\setminus A_{n-1}.\]
In words, an individual is newly informed at stage $n$ if it was an ignorant at stage $n-1$, and if it was within the radius of transmission of a spreader on its left. Once informed, an ignorant becomes a spreader and remains a spreader forever.
Then $\cup_{i\ge0}A_i$ is the set of spreaders (or informed individuals) at the end of the spreading procedure (stage $\infty$). Let $M:=|\cup_{i\ge0}A_i|$ be the final number of spreaders. Note that, in this case, $M$ coincides with the final range of the rumor. The event ``the rumor survives'' writes as
\[
\mathcal{A}:=\{M=\infty\}.
\]

Our first main result gives the exact probability of survival of the rumor.

\begin{theo}\label{theo:firework}
$$P(\mathcal{A})=\left(1+\sum_{j\ge1}\prod_{i=0}^{j-1}\alpha_i\right)^{-1}.\\$$
\end{theo}
As a direct corollary, we see that survival occurs with positive probability if, and only if, $\sum_{j\ge1}\prod_{i=0}^{j-1}\alpha_i<\infty$.
An important issue that, as far as we know, has not been addressed in previous works concerning this model, is bounds to the tail distribution of the final range of the rumor. This is the object of the following results. 

\begin{prop}\label{fire:Mdistr}
The random variable $M$ has finite expectation when $\prod_{k\ge0}\alpha_{k}>0$, and has exponential tail distribution when $\alpha_{k}$ increases exponentially fast to $1$.
\end{prop}

Under more specific assumptions, we can obtain more precise information on the tail distribution.

\begin{prop}\label{coro1}
 We have the following explicit bounds for the tail distributions.
\begin{enumerate}[(i)]
\item  If $1-\alpha_{k}\leq C_{r}r^{k},\,k\ge1$, for some $r\in(0,1)$ and a constant $C_{r}\in(0,\log \frac{1}{r})$ then
\[
P(M\ge k)\leq \frac1{C_{r}}(e^{C_{r}}r)^{k}.
\]
\item If $1-\alpha_k\sim(\log k)^{\beta}k^{-\alpha}$, $\beta\in \R$, $\alpha>1$, then there exists  $C>0$ such that, for large $k$'s, we have
$$P(M\ge k)\leq C(\log k)^{\beta}k^{-\alpha}.$$
\item If $1-\alpha_{k}= \frac{r}{k},\,k\ge1$ where $r\in(0,1)$, there exists  $C>0$ such that, for large $k$, we have
\[
P(M\ge k)\leq C\frac{(\ln k)^{3+r}}{(k)^{2-(1+r)^{2}}}.
\]
\item If $\alpha_k\sim ((k+1)/(k+2))^{\alpha}$,  $\alpha\in(1/2,1)$, then there exists $C=C(\alpha)>0$ such that, for large $k$, we have
\[
P(M\ge k)\leq \frac{C}{k^{1-\alpha}}.
\]
\end{enumerate}
\end{prop}
 
As examples, consider the following interesting variants of the model.

\begin{ex}\label{ex:geometric}
 Instead of having exactly one individual at each site of $\N$, suppose that there is an individual at each site with probability $\epsilon\in[0,1]$ independently of the other sites. 
 The FP considered in \cite{junior/machado/zuluaga/2011}
corresponds to the particular case where $\epsilon=1$. When $\epsilon < 1$, we obtain a  rumor process
in which the individuals are located at random positions and can be
arbitrarily far away to any other individual. In this sense,
$\epsilon$ is a ``sparseness'' parameter. This is a special case which is also studied by \cite[see Proposition 3.1 therein]{athreya/roy/sarkar/2004}.
 
In this variant, let ${\bf \bar{R}}=(\bar{R}_{i})_{i\ge1}$ be the  i.i.d. sequence of random radius, and let $P(\bar{R}_{0}=k)=\bar{\lambda}_{k}$ and $\bar{\alpha}_{k}:=P(\bar{R}_{0}\leq k)$.  
For any $i\ge0$, let $S_{i}$ be the Bernoulli random variable with parameter $\epsilon$ that indicates whether or not there is an individual at site $i$.
Let us define
\[
R_{i}:=\bar{R_{i}}.{\bf 1}\{S_{i} = 1\}
\]
which is the ``effective'' spreading radius  of site $i$: if there is nobody at site $i$ 
then the radius is $0$, because no sites on the right of $i$ are influenced by $i$. Otherwise, the radius is $\bar{R}_i$.
Thus, all the results stated above hold using
$$\alpha_k=P(R_0\leq k)= 1 - \epsilon ( 1 - \bar{\alpha}_k).$$
\end{ex}

\begin{ex}\label{ex:zucca}
Using different techniques \cite{bertacchi/zucca/2013} studied the following rumor
processes in random environment.  Consider
$(X_{n})_{n\geq0}$ a sequence of $\N$-valued i.i.d. random
variables and $(\bar{R}_{n}^{i})_{i\geq1,n \ge 1}$
a collection of i.i.d. random radius with $\bar{\alpha}_k = P(\bar{R}_n^i \leq k)$. At each site $n$, we have $X_n$ individuals, and each individual $i$ has a particular radius of spread $\bar R_{n}^{i}$. 

In order to apply our results to this model, 
let us define
\begin{equation}\label{eq:R}
R_n:=\sup_{i=1,\ldots,X_n}\bar{R}_{n}^i.
\end{equation}
We have $P(R_0\leq k)$ equals
\[
P\left(\sup_{i=1,\ldots,X_0}\bar{R}_{0}^i\leq k\right)=\sum_{l\ge1}P\left(\sup_{i=1,\ldots,l}\bar{R}_{0}^i\leq k\,,\,\,X_0=l\right)=\sum_{l\ge1}P(X_0=l)\bar{\alpha}_k^l.
\]
where the last equality follows from the independence among all the r.v.'s involved. Our results hold using
\[
\alpha_k=g_{X_0}(\bar{\alpha}_k)
\]
where $g_{X_0}(\cdot)$ is the probability generating function of $X_0$. In particular, Theorem 3.1 of \cite{bertacchi/zucca/2013} is a consequence of our Theorem \ref{theo:firework}. We further  obtain 
tail decays for the size of the set of spreaders. 
\end{ex}

\subsection{The Reverse Firework process}\label{sec:rfp}

Similarly to the previous section, we suppose that there is one individual at each site of $\N=\{0,1,\ldots\}$. In this model, the ignorant individuals take the information of a spreader within a random distance on its left. We will now let $B_n$, $n\ge0$ represents the set of  individuals that have been informed at stage $n$ in the Reverse Firework process. This sequence is also defined recursively through $B_0=\{0\}$ and, for $n\ge1$ \[B_n:=\{i\in\N:\,\,\textrm{there exists} \,\,j\in B_{n-1}\,\,\textrm{such that}\,\,j\in\{i-R_i,\ldots,i\}\}\setminus B_{n-1}.\]
In words, an individual is newly informed at stage $n$ if it was an ignorant at stage $n-1$, and if its radius covers a spreader on its left. Once informed, an ignorant becomes a spreader and remains in that state forever. Then $\cup_{i\ge0}B_i$ is the set of spreaders at the end of the spreading procedure (stage $\infty$) and we denote by $N$ its cardinality. As in the FP we define the  event ``the rumor survives'' 
by
\[
\mathcal{B}:=\{N=\infty\}.
\]

For this model, we obtain necessary and sufficient conditions for survival of the rumor as well as the  distribution of $N$ when the rumor dies out.

\begin{theo}\label{theo:revfirework}There exist two situations.
\begin{itemize}
\item If $\prod_{k\geq0}\alpha_{k} = 0$, then $P(\mathcal{B})=1$.
\item If $\prod_{k\geq0}\alpha_{k} >0$, then $P(\mathcal{B})=0$ and $N\sim\textrm{\emph{Geom}}\left(\prod_{k\geq0}\alpha_{k}\right)$.
\end{itemize}
\end{theo}

For any $n\ge1$, let $\zeta_n:={\bf1}\{n\in \cup_iB_i\}$, indicating whether the individual at site $n$ is a spreader or not at the end of the procedure. Let also $N(n):=\sum_{i=1}^{n}\zeta_i$ denotes the number of spreaders in $\{1,\ldots,n\}$.
We will now state limit theorems for the proportion of spreaders within $\{1,\ldots,n\}$, $N(n)/n$, when $n$ diverges.


Let
\begin{equation}
\label{eq:mu1}
\mu:=1+\sum_{j\ge1}\prod_{i=0}^{j-1}\alpha_i\,\,\,\,\,\,\,\textrm{and }\,\,\,\,\,\,
\sigma^2:= \sum_{k\ge1}k^2(1-\alpha_{k-1})\prod_{i=0}^{k-2}\alpha_i-\mu^2.
\end{equation}

\begin{theo}\label{theo:Nn}If $\mu<\infty$ then 
\[
\frac{N(n)}{n}\stackrel{a.s.}{\longrightarrow}\mu^{-1},
\]
and thus $\mu^{-1}$ is the final proportion of spreaders.
Moreover, if $\sigma^2 \in(0,\infty)$,
then
\[
\sqrt{n}\left(\frac{N(n)}{n}-\mu^{-1}\right)\stackrel{\mathcal{D}}{\rightarrow}\mathcal{N}\left(0,\frac{\sigma^2}{\mu^3}\right).
\]
 Otherwise, $N(n)/n \rightarrow 0$. 
\end{theo}

In particular, observe that according to Theorems \ref{theo:revfirework} and \ref{theo:Nn}, if the $\alpha_k$'s satisfy at the same time $\prod_k\alpha_k=0$ and $\mu=\infty$ (for instance, if they are as in items (iii) and (iv) of Proposition \ref{coro1}), then the information  reaches infinitely many individuals, but the final proportion of informed individuals is zero.

\begin{ex}\label{ex:3}
Theorem 4.1 of \cite{bertacchi/zucca/2013} follows  from Theorem \ref{theo:revfirework} by considering $R_k$ defined by \eqref{eq:R}. However, the distribution of $N$ and the limit theorems for $N(n)$ are a novelty. 
\end{ex}

\begin{ex}
Consider a model in which, if the nearest spreader on the left of site $i$ is at distance $k$ and $R_i\ge k$, then the individual at site $i$ believes the information with probability 
 $p_{k}$ where $p_{k}$ is a non-increasing sequence. Some examples:
\begin{itemize}
\item assuming $p_{k}=1$ for any $k$, we retrieve  the homogeneous case considered in \cite{junior/machado/zuluaga/2011},
\item assuming $p_{k}=\epsilon$ for any $k$, we obtain a model in which each individual is ``susceptible'' in that it believes the spreader within the radius on its left with a fixed probability $\epsilon$,
\item assuming $p_{k}\searrow0$, we convey the idea that the individual believes the nearest informed individual   in the radius on its left with a probability which decreases according to its distance. 
\end{itemize}
For this model, let ${\bf \bar{R}}=(\bar{R}_{i})_{i\in\mathbb{Z}}$ be the i.i.d. sequence of random radius, and let
$P(\bar{R}_{0}=k):=\bar{\lambda}_{k}$ and $\bar{\alpha}_{k}:=P(\bar{R}_{0}\leq k)$. This model is obtained as an example of the RFP, if we consider the i.i.d. sequence $(L_i)_{i\ge1}$ with $P(L_i\leq k)=1-p_k$, $k\ge0$ and we let
\[
R_i=\min\{\bar{R}_i,L_i\}
\]
denote the effective radius corresponding to this notion of  susceptibility. In this case, Theorem \ref{theo:revfirework} holds with
\[
\alpha_k=1-p_k(1-\bar{\alpha}_k).
\]
%
\end{ex}

\section{The discrete time renewal process}\label{sec:hoc}

The proofs of our results will be based on a remarkable relationship between the rumor processes introduced in the preceding section and a specific discrete time renewal process. This relationship  will be made explicit in Sections \ref{proof:1} and \ref{proof:2}. 
The present section is dedicated to define the  renewal process and list some of its properties. 
We will use, on purpose, the same notation as for the statements of the theorems.

Let  $(q_k)_{k\ge1}$ be a probability distribution on $\{1,2,\ldots\}\cup\{\infty\}$ defined by
\[
q_k=(1-\alpha_{k-1})\prod_{i=0}^{k-2}\alpha_i,
\]
and $q_{\infty} = 1 - \sum_{k} q_k$. Observe that, the mean and the variance of $(q_k)_{k\ge1}$ are given by (\ref{eq:mu1}).

Let ${\bf T} = (T_n)_{n\ge1}$ be an i.i.d. sequence of  r.v's, taking values in $\{1,2,\ldots\}\cup\{\infty\}$ with common distribution $(q_k)_{k\ge1}$. We call \emph{discrete (undelayed) renewal process} the process ${\bf Y}=(Y_n)_{n\ge0}$ defined through $Y_0=1$ and, for any $n\ge1$, $Y_n={\bf1}\{T_1+\ldots+T_i=n\,\,\textrm{for some}\,\,i\}$. Observe that $T_n$ is the distance between the $(n-1)^{\textrm{th}}$ and the $n^{\textrm{th}}$ occurrence of $1$ in ${\bf Y}$. As a consequence, $(q_k)_{k\ge1}$ is called the \emph{inter-arrival distribution}. Each occurrence of an $1$ is called a \emph{renewal}. Let $u_n:=\Pr(Y_n=1)$, $n\ge0$, be the corresponding discrete renewal sequence.

It is well-known that the chain ${\bf Y}$ is recurrent if, and only if, $\P(T=\infty)=\prod_{i\ge0}\alpha_i=0$ and,  in the recurrent regime, it is positive recurrent if, and only if, $\mu<\infty$. The number of $1$'s (number of renewals) occurring in ${\bf Y}$ up to time $n$, which we denote by $N(n)$,     satisfies the following  limit theorems \cite[Chapter 7]{Ross/2009}. If $\mu<\infty$, then $\frac{N(n)}{n}\stackrel{a.s.}{\rightarrow} \mu^{-1}$ and additionally, if $0<\sigma^2<\infty$, then
\[
\frac{N(n)-n\mu^{-1}}{\sqrt{n\sigma^2/\mu^3}}\stackrel{\mathcal{D}}{\rightarrow}\mathcal{N}(0,1).
\]

There are no simple explicit expressions for $u_n$, $n\ge1$. The well-known Discrete Renewal Theorem \cite[Chapter 7]{Ross/2009} states that $u_k\rightarrow\mu^{-1}$ and some results give information concerning the rate at which this convergence occurs. For instance, for the case where $\mu=\infty$,  the following proposition is due to \cite{bressaud/fernandez/galves/1999a}. 

\begin{prop}\label{prop:hoc}
When $\mu=\infty$, $u_k$ converges to zero at
\begin{enumerate}[(i)]
\item summable rate, if $\prod_{i\ge0}\alpha_i>0$ (that is, if $1-\alpha_k$ summable);
\item exponentially rate, if $1-\alpha_k$ decreases exponentially to $0$.
\end{enumerate}
\end{prop}

 The next proposition gives more explicit estimates under more specific assumptions. Items (i) and (iii) are due to \citet[Proposition B.2]{gallo/lerasle/takahashi/2013}, item (ii) is due to \citet[Remark 5]{bressaud/fernandez/galves/1999a} and item (iv) is due to \citet[Theorem 1.1]{garsia/lamperti/1962}.
 \begin{prop}\label{prop:expli}
We have the following explicit upper bounds.
\begin{enumerate}[(i)]
\item  If $1-\alpha_{k}\leq C_{r}r^{k},\,k\ge1$, for some $r\in(0,1)$ and a constant $C_{r}\in(0,\log \frac{1}{r})$ then
\[
u_k\leq \frac1{C_{r}}(e^{C_{r}}r)^{k}.
\]

\item If $\prod_{i\ge0}\alpha_i>0$ and $\sup_{j}\limsup_{k\rightarrow+\infty}(\frac{1-\alpha_j}{1-\alpha_{kj}})^{1/k}\leq1$, then there exists a constant $C>0$ such that, for large $k$, $u_{k}\leq C(1-\alpha_k)$.
\item If $\alpha_k= \frac{r}{k}+s_{k},\,k\ge1$ where $r\in(0,1)$ and $\{s_{n}\}_{n\ge1}$ is a summable sequence, there exists a constant $C>0$ such that
\[
u_{k}\leq C\frac{(\ln k)^{3+r}}{(k)^{2-(1+r)^{2}}}.
\]
\item If $\prod_{i\ge k+1}\alpha_{i}=L(k)k^{-\alpha}$ where $L(k)>0$ and $\frac{L(\lambda k)}{L(k)}\rightarrow1$ for any $\lambda>0$ and $1/2<\alpha<1$, then, there exists $C(\alpha)>0$ such that ,for large $k$,
\[
u_k\sim \frac{C(\alpha)}{k^{1-\alpha}L(k)}.
\]
\end{enumerate}
\end{prop}

\section{Proofs}
In this section, we construct the FP and the RFP through a sequence
${\bf U}=(U_n)_{n \in \mathbb{Z}}$ of iid r.v.'s uniformly distributed in $[0,1[$.  Let
    $\mathbb{P}$ denotes the product law of ${\bf U}$.
At each $i\ge1$, $U_i$ is used to specify the random radius
\[
R_{i}:=\sum_{k\ge0}k{\bf 1}\{\alpha_{k-1}\leq U_{i} <\alpha_{k}\}\,\,\,\textrm{where $\alpha_{-1}:=0$}.
\]
We recall that, for the FP process, this is  the radius at which the individual at site $i$ transmits the information on its right and, for the RFP, this is the radius at which the individual at site $i$ takes the information on its left.

\subsection{Firework Process}\label{proof:1}

The  proofs are based on the following lemma.
\begin{lemma}\label{lemma:1}
For any $n\ge0$, we have $P(M>n)={u}_{n+1}$.
\end{lemma}
\begin{proof}The first key point is to observe that $M$ can be written as follows
\begin{align}
M&=\min\{i\ge0:R_{j}\le i-j\,,\,\,j=0,\ldots,i\}\\
&=\min\{i\ge0:U_{j}< \alpha_{i-j}\,,\,\,j=0,\ldots,i\}.
\end{align}
However,  the proof will be simpler if we work with the \emph{reversed} random variable
\[
\bar{M}:=\max\{i\le0:U_{j}< \alpha_{j-i}\,,\,\,j=i,\ldots,0\},
\]
which satisfies $-\bar{M}\stackrel{\mathcal{D}}{=}M$. Thus, what we have to prove is that  $\mathbb{P}(\bar{M}<-n)={u}_{n+1}$.
The second key point is to observe that this definition of $\bar{M}$ is similar to the definition of $\tau[0]$ considered in \cite{comets/fernandez/ferrari/2002} (see display (4.2) therein). The proof of our lemma would then, follow from a direct analogy with display (5.6) therein. We nevertheless  include all the details here for completeness. 

Let $({\bf H}^{(m)})_{m\in\mathbb{Z}}$ be a family of Markov processes, where the index $m$ indicates where each one  starts, defined recursively  using the single sequence ${\bf U}$ as follows. For any $m\in\Z$, put $H^{(m)}_m=0$ and
\[
H_n^{(m)}=(H_{n-1}^{(m)}+1){\bf1}\{U_n<\alpha_{H_{n-1}^{(m)}}\},\,\,\,n> m.
\]
The corresponding  transition matrix $Q$ is such that
$Q(i,i+1)=\alpha_i$ and $Q(i,0) = 1-\alpha_i$. Since ${\bf H}^{(m)}$ is a Markov chain, it renews at each visit to
$0$ and  the distance
between two successive visits to $0$ has distribution
$q_k=(1-\alpha_{k-1})\prod_{i=0}^{k-2}\alpha_i$, where
$\prod_{i=0}^{k-2}\alpha_i$ means that the chains climbs up from $0$
to $k-1$, and $(1-\alpha_{k-1})$ means that it falls down to
$0$. Consequently, for any $m\in\mathbb{Z}$ and $k\ge1$, we have
$\mathbb{P}(H_{m+k}^{(m)}=0)=u_k$.  

This family of coupled Markov
processes has two important properties.
\begin{enumerate}
\item \emph{Monotonicity}:
\[
H_{n}^{(m)}\ge H_{n}^{(k)}\,,\,\,\forall\,m<k\leq n,
\]
which implies in particular that $H_{n}^{(m)}=0\Rightarrow H_{n}^{(k)}=0$ for all $m<k\leq n$. 
\item \emph{Coalescence at $0$}, that is
\[
H_{n}^{(m)}=0\Rightarrow H_{t}^{(m)}=H_{t}^{(k)}\,,\,\,\forall\,m\leq k\leq n\leq t.
\]
\end{enumerate}
Using these properties, we  obtain the following sequence of equivalences, for any $j\le 0$:
\begin{align*}
\bar{M}<-n&\Leftrightarrow \forall i\in\{-n,\ldots,0\},\,\exists j\in\{i,\ldots,0\}:\,\,U_{j}> \alpha_{i-j}\\
&\Leftrightarrow  \forall i\in\{-n,\ldots,0\},\,\exists j\in\{i,\ldots,0\}:\,\,H^{(i-1)}_{j}=0\\
&\Leftrightarrow \forall i\in\{-n,\ldots,0\},\,\,H^{(i-1)}_{0}=0\\
&\Leftrightarrow H^{(-n-1)}_{0}=0,
\end{align*}
where the first line follows from the definition of $\bar{M}$, the second line follows from the definition of the family of Markov processes, the third line follows from the coalescing property, and the forth line follows from the monotonicity.

We therefore obtained that $\mathbb{P}(\bar{M}<-n)=\mathbb{P}(H^{(-n-1)}_{0}=0)=\mathbb{P}(H^{(0)}_{n+1}=0)$.  Thus $\mathbb{P}(\bar{M}<-n)=u_{n+1}$. 
\end{proof}

 Theorem \ref{theo:firework} and Proposition \ref{fire:Mdistr} follow directly
 from Lemma \ref{lemma:1}, the fact that $u_k \rightarrow \mu^{-1}$
 and Proposition \ref{prop:hoc}. Proposition
 \ref{coro1} follows from Proposition \ref{prop:expli}
 by simple calculations.

\subsection{Reversed Firework Process}\label{proof:2}

Recall the definition of the sequence of the sets $B_n$, $n\ge1$. For any
$t,n\ge1$, let $\zeta_n(t):={\bf 1}\{n\in \cup_{j\leq t}B_j\}$ and
observe that $\zeta_{n}(t)$ is non-decreasing in $t$ for each fixed
$n\ge1$.  It follows that, by monotonicity, when $t$ goes to infinity
the sequence of processes $({\boldsymbol \zeta}(t))_{t\ge1}$
converges weakly to the process ${\boldsymbol \zeta}$ introduced in
Section \ref{sec:rfp}.

%

The proofs of the results are based on the following lemma.

\begin{lemma}\label{lemma:key2} ${\boldsymbol \zeta}\stackrel{\mathcal{D}}{=}{\bf Y}$.
\end{lemma}

\begin{proof}[Proof of Lemma \ref{lemma:key2}]
For any sequence $a_{m}^{n}\in \{0,1\}^{n-m+1}$, $-\infty\leq m\leq n<+\infty$, we define
\[
\ell(a_{m}^{n}):=\inf\{i\ge0: \,a_{n-i}=1\},
\]
which is the number of zeros after the last occurrence of $1$ in $a_{m}^{n}$. We use the convention that $\ell(a_{m}^{n})=\infty$ when $a_{i}=0$ for $i=m,\ldots,n$. 
We have
\begin{align}\label{eq:conta}
\{\zeta_{n}=1\}&=\bigcup_{t\ge1}\{\zeta_{n}(t)=1\}=\bigcup_{t\ge1}\{R_{n}\ge\ell(\zeta_{0}^{n-1}(t))\}=\{R_{n}\ge\ell(\zeta_{0}^{n-1})\}
\end{align}
where we used the fact that $\zeta_{i}(t)$ is non-decreasing in $t$ in the first and in the last equalities. Observe that since $\zeta_{0}=1$, we always have $\ell(\zeta_{0}^{n-1})\leq n-1$.  Since $U_{n}$ is  independent of $\mathcal{F}(U_{1}^{n-1})$ with respect to which $\zeta_{0}^{n-1}$ is measurable, it follows that
\begin{align*}
\mathbb{P}(\zeta_{n}=1|\zeta_{0}^{n-1}=a_{0}^{n-1})&=\mathbb{P}(R_{n}\ge\ell(a_{0}^{n-1}))=1-\alpha_{\ell(a_{0}^{n-1})}.
\end{align*}
In other words, in ${\boldsymbol \zeta}$, the conditional probabilities with respect to the ``past'' only depend on the distance until the last occurrence of an $1$ (that is, the nearest occurrence of an $1$ on the left).  It follows that ${\boldsymbol \zeta}$ is an $1$ (at site $0$) followed by a concatenation of i.i.d. blocks of random length $K$ of the form $0^{K-1}1$. In other words, it is a renewal process with inter-arrival distribution $P(K=k)$, $k\ge1$. Moreover, observe that $P(K=k)=(1-\alpha_{k-1})\prod_{i=0}^{k-2}\alpha_i$, which directly follows from 
\begin{align*}
\mathbb{P}(\zeta_{1}^{k}=0^{k-1}1|\zeta_{0}=1)&=\prod_{i=1}^{k}P(\zeta_i=0|\zeta_{0}^{i-1}=10^{i-1})\times \mathbb{P}(\zeta_k=1|\zeta_{0}^{k-1}=10^{k-1}).
\end{align*}
In other words, ${\boldsymbol\zeta}$ is a renewal process with the same inter-arrival distribution as ${\bf Y}$, and thus, they have the same distribution as claimed.
\end{proof}
The proof of Theorem \ref{theo:Nn} and of most of the statements of Theorem \ref{theo:revfirework} follow directly from Lemma \ref{lemma:key2} and the results of Section \ref{sec:hoc} concerning ${\bf Y}$. 

The only missing  statement of Theorem \ref{theo:revfirework}  is that  $N\sim\textrm{\textrm{Geom}}(r)$ when $r:=\prod_k\alpha_k>0$. But this can be seen from the fact that  ${\bf Y}$ renews at each visit to $1$, and that at each visit, it has probability $r$ of never coming back to $1$.

\section{Discussion and possible extensions} 

Here we list some interesting observation and possible extensions that are under consideration in works in progress. 
\begin{enumerate}
\item Consider the FP and the RFP running with the same $\alpha_k$'s. An interesting observation that may not be obvious at first glance is that the probability that the information reaches the individual at site $n$ is equal in both processes. This is clear from the proofs (Lemmas \ref{lemma:1} and \ref{lemma:key2}).
\item A first natural extension is to consider the models where the individuals propagate or take  the information on both sides (with same radii).
\begin{itemize}
\item Obviously, for the FP, this change only makes sense if individuals are disposed on $\Z$, since on $\N$ all the results are the same. 
\item For the RFP, the question still makes sense on $\N$. In fact, if the radius goes on both sides in the RFP on $\N$, conditions for survival are the same as in the original RFP, but results concerning the proportion of informed individuals will change.
\end{itemize} 
\item Example \ref{ex:geometric} considers the case where the individuals are disposed according to an i.i.d. process (at each site, 
 the probability that there is an individual is $\epsilon$, independently of the other individuals). We see that the independence was crucial, as it allows us to compare each model with the original model with a new sequence of i.i.d. \emph{effective radius}.
 \begin{itemize}
 \item 
\cite{athreya/roy/sarkar/2004} studied a case similar to Example \ref{ex:geometric}, but where  
 the sequence of individuals are disposed according to a Markov process. This case  is not covered by Theorem \ref{theo:firework}.
 As they only obtain sufficient conditions for survival or not of the rumor, it is natural to wonder whether the results that we obtain in the i.i.d. case are also valid in the Markovian case.  
  \item A further  natural extension would be to consider the case where the individuals are disposed according to an arbitrary renewal process. A similar extension could be studied for the RFP as well. 
 \end{itemize}
 \end{enumerate}

\section*{Acknowledgements}
This work was partially funded by CNPq grants 479313/2012-1, 302755/2010-1 and FAPESP grant 2013/03898-8. The authors thank ICMC/USP and IMECC/UNICAMP for their hospitality.

\bibliographystyle{jtbnew}
\bibliography{sandrobibli}

\end{document}